%% file: Nonamenable_BRW_Revised2.tex
\documentclass[11pt, reqno]{article}
\input{header}

\input{commands}

\usepackage[margin=1.08in]{geometry}
\usepackage{extarrows}
\usepackage{titling}
\usepackage{commath}
\usepackage{wasysym}
\usepackage{mathtools}
\usepackage{titlesec}

\usepackage{bbm}
\usepackage{setspace}
\usepackage{enumitem}

\usepackage[textsize=tiny]{todonotes}

\usepackage{cite}

\newcommand{\Aut}{\operatorname{Aut}}

\newcommand{\bP}{\mathbf P}
\newcommand{\bE}{\mathbf E}

\newcommand{\myfrac}[3][0pt]{\genfrac{}{}{}{}{\raisebox{#1}{$#2$}}{\raisebox{-#1}{$#3$}}}

\def\P{\mathbb{P}}
\DeclareMathSymbol{\leqslant}{\mathalpha}{AMSa}{"36} % nicer `smaller or equal'
\DeclareMathSymbol{\geqslant}{\mathalpha}{AMSa}{"3E} % nicer `larger or equal'
\DeclareMathSymbol{\eset}{\mathalpha}{AMSb}{"3F}     % nicer `emptyset'
\renewcommand{\epsilon}{\varepsilon}

\pretitle{\begin{flushleft}\Large}
\posttitle{\par\end{flushleft}\vskip 0.5em
}
\preauthor{\begin{flushleft}}
\postauthor{
\\
\vspace{0.5em}
\footnotesize{Statslab, DPMMS, University of Cambridge.\\ Email: 
\href{mailto:t.hutchcroft@maths.cam.ac.uk}{t.hutchcroft@maths.cam.ac.uk}}
\end{flushleft}}
\predate{\begin{flushleft}}
\postdate{\par\end{flushleft}}
\title{\bf Non-intersection of transient branching random walks}

\renewenvironment{abstract}
 {\par\noindent\textbf{\abstractname.}\ \ignorespaces}
 {\par\medskip}

\titleformat{\section}
  {\normalfont\large \bf}{\thesection}{0.4em}{}

\author{{\bf Tom Hutchcroft}}

\begin{document}

\date{\small{\today}}

\maketitle

\setstretch{1.1}

\begin{abstract}
Let $G$ be a Cayley graph of a nonamenable group with spectral radius $\rho < 1$. It is known that branching random walk on $G$ with offspring distribution $\mu$ is \emph{transient}, i.e., visits the origin at most finitely often almost surely, if and only if the expected number of offspring $\overline \mu$ satisfies $\overline \mu \leq \rho^{-1}$. Benjamini and M\"uller (\emph{Groups Geom.\ Dyn.}, 2010) conjectured that throughout the transient supercritical phase $1<\overline{\mu} \leq \rho^{-1}$, and in particular at the recurrence threshold $\overline \mu = \rho^{-1}$, the trace of the branching random walk is tree-like in the sense that it is infinitely-ended almost surely on the event that the walk survives forever. This is essentially equivalent to the assertion that two independent copies of the branching random walk intersect at most finitely often almost surely. 
We prove this conjecture, along with several other related conjectures made by the same authors. 

A central contribution of this work is the introduction of the notion of \emph{local unimodularity}, which we expect to have several further applications in the future.

% A

% with positive probability if and only if $\lambda > \rho^{-1}$. We prove that 

% As a part of our proof, we study how the recurrence or transience of a set for branching random walk depends on the offspring distribution.
\end{abstract}

\section{Introduction}

Let $G=(V,E)$ be a connected, locally finite graph.
\emph{Branching random walk} on $G$  is a Markov process taking values in the space of finitely-supported functions $V\to \{0,1,2,\ldots\}$, which we think of as encoding the number of particles occupying each vertex of $G$. We begin with a single particle, which occupies some vertex $v$. At every time step, each particle splits into a random number of new particles according to a fixed offspring distribution $\mu$, each of which immediately performs a simple random walk step on $G$.
% , i.e., chooses an edge emanating from its current position uniformly at random and crosses this edge.
 Equivalently, branching random walk can be described as a  random walk on $G$ indexed by a Galton-Watson tree \cite{MR1258875,MR1254826}. 
 % That is, we can first sample the genealogical tree of the particles, and then choose a uniformly random embedding of this tree into $G$ sending the root of the tree to $v$. 
 We say that the offspring distribution $\mu$ is \textbf{non-trivial} if $\mu(1)<1$. It follows from the classical theory of branching processes (see e.g.\ \cite[Chapter 5]{LP:book}) that branching random walk exhibits a \emph{phase transition}: If the mean offspring $\overline{\mu}$ satisfies $\overline{\mu} > 1$ then the process survives forever with positive probability, while if $\mu$ is non-trivial and $\overline{\mu} \leq 1$ then the process survives for only finitely many time steps almost surely.

% Indeed, 
Beyond its intrinsic appeal and its function as a model for many processes appearing in the natural sciences, branching random walk also attracts attention as a toy model that lends insight into more complex processes. Indeed, many models of statistical mechanics are expected to have \emph{mean-field behaviour} in high dimensions, which roughly means that their behaviour at criticality is similar to that of a critical branching random walk.  Mean-field behaviour has now been proven to hold in high dimensions for percolation \cite{MR1171762,heydenreich2015progress}, the Ising model \cite{MR857063}, the contact process \cite{MR1851386}, uniform spanning trees \cite{hutchcroft2018universality,peres2004scaling}, and the Abelian sandpile model \cite{hutchcroft2018universality}, among other examples. When comparing branching random walk to these models, the main questions of interest often concern the geometric properties of the \emph{trace} of the branching random walk, i.e., the subgraph of $G$ spanned by the set of edges that are ever crossed by some particle.
% More precisely, we fix a connected, locally finite graph $G=(V,E)$ and an offspring distribution $\mu$ (i.e., a probability measure on $\{0,1,2,\ldots\}$).

Although branching random walks have traditionally been studied primarily in the case of Euclidean lattices such as $\Z^d$, it is natural to consider such processes on more general graphs. Recall that a graph $G$ is said to be \textbf{nonamenable} if its \textbf{spectral radius}\footnote{In the abstract, we wrote $\rho$ for the spectral radius to avoid introducing the notation $P$. We will now switch to the notation $\|P\|$ to avoid conflicts with the standard notation for the root of a unimodular random rooted graph.}
\[
\|P\| = \lim_{n\to\infty} p_{2n}(v,v)^{1/2n}
\]
is strictly less than $1$. Here, $p_n(u,v)$ denotes the probability that a simple random walk on $G$ started at $u$ is at $v$ after $n$ steps, and $P(u,v)=p_1(u,v)$ is the associated Markov operator.
See e.g.\ \cite[Chapter 6]{LP:book} for background on amenability and nonamenability. 
Branching random walk is particularly interesting on nonamenable graphs as it exhibits a \emph{double phase transition} \cite{MR1254826,MR2284404,MR2426846}: Suppose that $\mu$ is non-trivial. If $0\leq \overline{\mu}\leq 1$ then the process dies after finite time almost surely, if $1 < \overline{\mu} \leq \|P\|^{-1}$ then the process survives forever with positive probability but does not visit any particular vertex infinitely often almost surely, while if $\overline{\mu}>\|P\|^{-1}$ then the process has a positive probability to return to its starting point infinitely often.
 % vertex infinitely often almost surely on the event that it survives. 
When $G$ is a \emph{transitive} nonamenable graph, such as a  Cayley graph of a nonamenable group, rather more is known: For $\overline{\mu}>\|P\|^{-1}$ the branching random walk visits every vertex infinitely often \emph{almost surely} on the event that it survives forever \cite[Lemma 5.1]{MR1254826}, while for $\overline{\mu}\leq \|P\|^{-1}$ the \emph{expected} number of times the walk returns to the origin is finite \cite[Theorem 7.8]{Woess}. We say that a branching random walk is \textbf{transient} if it visits every vertex at most finitely often almost surely.

 These facts are analogous to the conjectured existence of a non-uniqueness phase for Bernoulli percolation on nonamenable groups \cite{bperc96}. The interested reader is referred to \cite{1804.10191,MR2280297} and references therein for background on this conjecture. Moreover, it is hoped that studying the behaviour of branching random walk at and near the recurrence threshold $\overline{\mu}=\|P\|^{-1}$ will yield insight into the behaviour of percolation at and near the \emph{uniqueness threshold}, $p_u$, a topic that remains very poorly understood in general. Indeed, it would be very interesting to develop a mean-field theory of percolation at $p_u$ and give conditions under which it can be compared, in some sense, to branching random walk at the recurrence threshold. See \cite[Section 6.2]{1804.10191} for potential avenues of  research in this direction. 
% In order to develop and apply such a theory, it is desirable to better understand the behav
 Most existing work regarding branching random walk at the recurrence threshold has focused on the case of trees and Gromov hyperbolic groups, the theory of which is now rather sophisticated \cite{MR1452555,MR3087391,MR3194496,MR1736590}. See also \cite{MR2946082} for some related results on free products.

In \cite{MR2914859}, Benjamini and M\"uller studied the geometry of the trace of the branching random walk in the transient supercritical regime $1<\overline{\mu} \leq \|P\|^{-1}$ on general nonamenable Cayley graphs, and posed a large number of questions about this geometry. One of the most interesting of these questions \cite[Question 4.1]{MR2914859} asked whether the trace of branching random walk  throughout the transient regime is tree-like in the sense that it is infinitely-ended with no isolated ends. Here, we recall that, for $k\in \N \cup \{\infty\}$, an infinite graph is said to be \textbf{$k$-ended} if deleting a finite set of vertices from the graph results in a supremum of $k$ infinite connected components. A graph has \textbf{no isolated ends} if every infinite connected component that remains after the removal of finitely many vertices from the graph is itself infinitely-ended. See e.g.\ \cite[Section 21]{Woess} for a more systematic development of these notions.
% One can define the \emph{space of ends} of $G$, a topological space that has cardinality $k$ if and only if $G$ is $k$-ended, see \cite{Woess} for further background. In particular, it follows from classical point-set topology that the space of ends of a locally finite graph with no isolated ends has the cardinality of the continuum.

Partial progress on this question was made by Gilch and M\"uller \cite{MR3644010}, who studied planar hyperbolic Cayley graphs, and Candellero and Roberts \cite{MR3333735}, who studied graphs satisfying the condition $\sum_{n\geq 1} n \|P\|^{-n}p_n(v,v) < \infty$. 
% We expect that it is also possible to treat the general Gromov hyperbolic case using the results of Gou\"ezel \cite{MR3194496}. 
Both of these results  rely on methods that are quite specific to the examples they treat, and many interesting cases were left open. 
% The primary purpose of this paper is to resolve the question in the full generality. 
In this paper we resolve the question in full generality.

 % and asked several questions 

\begin{theorem}
\label{thm:ends}
Let $G$ be a unimodular transitive graph.
 % and let $P$ be a symmetric, $\Aut(G)$-invariant transition matrix on $G$ with $\rho(P)<1$. 
 Let $\mu$ be an offspring distribution with $1<\overline{\mu} \leq \|P\|^{-1}$. Then the trace of a branching random walk on $G$ with offspring distribution $\mu$ is infinitely ended and has no isolated ends almost surely on the event that it survives forever.
\end{theorem}

Here, \emph{unimodularity} is a technical condition that holds for every Cayley graph of a finitely generated group \cite{MR1082868} and that is introduced in detail in \cref{sec:further}. 

\medskip

We also resolve several further questions raised in \cite{MR2914859} in \cref{sec:further}, namely \cite[Conjecture 4.1, Conjecture 4.2, and Question 4.5]{MR2914859}.

\medskip

We will deduce \cref{thm:ends} as an easy corollary of the following more fundamental theorem concerning the intersection of the traces of two independent branching random walks. 
We use `$\mu$-BRW on $G$' as shorthand for `branching random walk on $G$ with offspring distribution $\mu$'.

\begin{theorem}
\label{thm:nonintersection}
Let $G$ be a unimodular transitive graph. Let $\mu_1,\mu_2$ be non-trivial offspring distributions with $\overline{\mu_1}, \overline{\mu_2} \leq \|P\|^{-1}$, and let $x$ and $y$ be vertices of $G$. Then an independent $\mu_1$-BRW started at $x$ and $\mu_2$-BRW  started at $y$ intersect at most finitely often almost surely.
\end{theorem}

In other words, under the hypothesis of the theorem, there are almost surely at most finitely many vertices of $G$ that are visited by both branching random walks.
 % Note that we do \emph{not} require the two processes to visit the same vertex at the same time.
  (The result is very easy when the strict inequalities $\overline{\mu_1}, \overline{\mu_2} < \|P\|^{-1}$ hold, see \cref{lem:subcritical}.) 

\medskip

We remark that the study of the intersections of two \emph{simple} random walks is a classical topic, first studied by Erd\"os and Taylor \cite{ErdTay60}, with close connections to the uniform spanning tree. See e.g.\ \cite[Section 10.5]{LP:book} and references therein for more on this topic. See also \cite{MR2308594} for results on the geometry of \emph{simple} random walk traces, and \cite{MR3231996,BJKS08,BeCu12} for results on the geometry of the trace of critical branching random walk conditioned to survive forever on $\Z^d$.

\medskip

\noindent
\textbf{About the proof.}
The proof of \cref{thm:ends,thm:nonintersection} takes a rather different approach than has previously been taken in the literature. A central contribution is the notion of \emph{local unimodularity}, which we introduce in \cref{sec:localunimodularity}. The relevance of this notion to our setting is established by \cref{prop:unimodularpush}, which allows us to `push forward' and `pull back' local unimodularity through tree-indexed walks. We then formulate and prove a version of the \emph{Magic Lemma} of Benjamini and Schramm \cite{BeSc}. Intuitively, this lemma states that for any finite set of vertices $A$ in a tree $T$, the set $A$ `looks like it accumulates to at most two ends of $T$' from the perspective of a uniformly random element of $A$. (The original Magic Lemma concerns finite sets of points in $\R^d$; the statement about trees that we use is closely related and is implicit in the original proof.)
Let $T$ be a Galton-Watson tree, let $X$ be a tree-indexed walk on $G$ indexed by $T$, and let $I$ be the set of vertices of $T$ that get mapped by $X$ into the trace of an independent branching random walk on $G$. 
Using our formulation of the Magic Lemma together with the local unimodularity result \cref{prop:unimodularpush}, we are able to prove that the set $I$ is either finite or accumulates to at most two ends of $T$ almost surely. 
% This is done by approximating all Galton-Watson  trees involved by Galton-Watson trees with
The latter possibility is easily ruled out using the Markovian nature of branching random walk, completing the proof.

\begin{remark}
The proof of \cref{thm:ends,thm:nonintersection} admits various generalizations. For example, one can allow the two branching random walks to have different (possibly long-range) step distributions, provided that both associated transition matrices are symmetric and invariant under the diagonal action of the automorphism group in the sense that $P(\gamma x, \gamma y)=P(x,y)$ for every $x,y \in V$ and $\gamma \in \Aut(G)$. One could also consider branching random walks in random environment, provided that this random environment is almost surely nonamenable, has law invariant under the automorphism group of $G$, and is such that the stationary measure of the root has finite second moment: A simple example is given by assigning  i.i.d.\ random conductances taking values in $[1,2]$ to the edges of $G$. Even more generally, one could consider branching random walks on unimodular random rooted networks that are almost surely nonamenable and for which the conductance of the root has finite second moment. The details of these generalizations are straightforward, and we restrict attention to the above case for clarity of exposition.
\end{remark}

\section{Background on unimodularity}

\label{sec:background}

% In addition to our main results, we also resolve several other questions and conjectures appearing in \cite{MR2914859}, namely Conjecture 4.1, Conjecture 4.2, and Question 4.5 of that paper. 
% These results are stated in terms of \emph{unimodular random rooted graphs}, the definition of which we now briefly recall.
We now briefly recall the definition of \emph{unimodular random rooted graphs} and some basic facts about them. 
These definitions were first suggested by Benjamini and Schramm \cite{BeSc} and were developed systematically by Aldous and Lyons \cite{AL07}. A detailed and readable introduction can be found in \cite{CurienNotes}. Unimodularity has been found to often lead to surprisingly simple, conceptual, and generalisable solutions to problems that can appear intractable from a  classical perspective, see e.g.\ \cite{BeSc,HutPe2015a,MR3983336,MR3297773,MR3863918}.
% \subsection{Extensions}

% The appropriate setting for our result will be related to \emph{unimodular random rooted graphs}, which we now introduce. 
 A \textbf{rooted graph} $(g,u)$ is a connected, locally finite graph $g=(V(g),E(g))$ together with a distinguished vertex $u$, the root. (We will often use the convention of using lower case letters for deterministic rooted graphs and upper case letters for random rooted graphs.) An isomorphism of graphs is an isomorphism of rooted graphs if it preserves the root. 
We denote the space of isomorphism classes of rooted graphs by $\cG_\bullet$. 
(We will ignore the distinction between a rooted graph and its isomorphism class when this does not cause confusion.)
This space carries a natural topology, known as the \emph{local topology}, in which two rooted graphs are close if there exist large balls around their respective roots that are isomorphic as rooted graphs. A \emph{doubly-rooted} graph is a connected, locally finite graph together with an \emph{ordered pair} of distinguished vertices. The space of isomorphism classes of doubly-rooted graphs $\cG_{\bullet\bullet}$ and the local topology on this space are defined similarly to the singly-rooted case. Both $\cG_\bullet$ and $\cG_{\bullet\bullet}$ are Polish spaces. See \cite[Section 2.1]{CurienNotes} for details. We write $\cT_\bullet$ and $\cT_{\bullet\bullet}$ for the closed subspaces of $\cG_\bullet$ and $\cG_{\bullet\bullet}$ in which the underlying graph is a tree.

We call a random variable taking values in $\cG_\bullet$ a \emph{random rooted graph}. A random rooted graph $(G,\rho)$ with vertex set $V$ is said to be \emph{unimodular} if it satisfies the \textbf{mass-transport principle}, which states that
\begin{equation}
\label{eq:MTP}
\E\left[ \sum_{v\in V} F(G,\rho,v) \right] = \E\left[ \sum_{v\in V} F(G,v,\rho) \right] 
\end{equation}
for every measurable function $F:\cG_{\bullet\bullet}\to[0,\infty]$. We call a \emph{probability measure} $\mu$ on $\cG_\bullet$ unimodular if a random rooted graph with law $\mu$ is unimodular. The set $\cU(\cG_\bullet)$ of unimodular probability measures on $\cG_\bullet$ is a weakly closed, convex subset of the space of all probability measures on $\cG_\bullet$ \cite[Theorem 8]{CurienNotes}. We think of $F$ as a rule for sending a non-negative amount of mass $F(G,u,v)$ from $u$ to $v$: the mass-transport principle states that the expected amount of mass the root receives is equal to the expected amount of mass it sends out. Intuitively, $(G,\rho)$ is unimodular if the root $\rho$ is `uniformly distributed on the vertex set of $G$'. Although this statement cannot be interpreted literally when $G$ is infinite, it remains very useful as a heuristic.

% \medskip

 A transitive graph $G$ is said to be unimodular if $(G,\rho)$ is a unimodular random rooted graph whenever $\rho$ is an arbitrarily chosen root vertex of $G$. Every amenable transitive graph and every Cayley graph of a finitely generated group is unimodular \cite{MR1082868}. 

% \medskip

It will be convenient for us to introduce the following more general notion.
% Given a positive measurable function $W:\cG_\bullet^\diamond\to (0,\infty)$,
 We say that a random rooted graph $(G,\rho)$ is \textbf{quasi-unimodular} if there exists a measurable function $W:\cG_\bullet \to (0,\infty)$ such that $\E[W(G,\rho)]=1$ and
\[
\E\left[W(G,\rho)\sum_{v\in V} F(G,\rho,v)\right]
=
\E\left[W(G,\rho)\sum_{v\in V} F(G,v,\rho)\right]
\]
for every measurable function $F:\cG_{\bullet\bullet} \to [0,\infty]$; in this case we say that $(G,\rho)$ is quasi-unimodular with weight $W$. 
% if it is locally $W$-unimodular for some $W$.
Equivalently, $(G,\rho)$ is quasi-unimodular if and only if there exists a  unimodular random rooted graph  $(G',\rho')$ whose law is equivalent to that of $(G,\rho)$ in the sense that both measures are absolutely continuous with respect to each other: the weight $W$ is the Radon-Nikodym derivative of the law of $(G',\rho')$ with respect to the law of $(G,\rho)$. Thus, for most qualitative purposes, being quasi-unimodular is just as good as being unimodular.

\begin{remark}
\label{remark:W_uniqueness}
A notion closely related to that of quasi-unimodularity is studied  under the name \emph{unimodularizability} by Khezeli \cite{MR3880017}, who shows in particular that the weight $W$ is unique up to a factor that depends only on the invariant $\sigma$-algebra \cite[Theorem 3]{MR3880017}. In particular, the weight $W$ is unique (up to a.e.-equivalence) if $(G,\rho)$ is ergodic. We will not require this result.
\end{remark}

% \medskip

The following proposition allows us to obtain new quasi-unimodular random rooted graphs as traces of unimodular random rooted trees. See \cite{MR1258875} for detailed definitions of Markov chains indexed by trees. We say that a tree-indexed walk is \textbf{transient} if it visits every vertex at most finitely often almost surely.
Here and elsewhere, we write either $\deg_G(v)=\deg(v)$ for the degree of a vertex $v$ in the graph $G$, using the subscript only if the choice of graph is ambiguous. 
Recall that $\operatorname{Tr}(X)$ is defined to be the subgraph of $G$ spanned by every edge that is ever crossed by $X$.

\begin{prop}
\label{prop:unimodular_trace}
Let $(G,\rho)$ be a unimodular random rooted graph, and let $(T,o)$ be an independent unimodular random rooted tree. Let $X$ be a $T$-indexed walk in $G$ with $X(o)=\rho$, and let $\operatorname{Tr}(X)$ be the trace of $X$. Suppose that $X$ is almost surely transient and that the integrability assumption $\E[\deg_G(\rho)(\#X^{-1}(\rho))^{-1}]<\infty$ holds.  Then $(\operatorname{Tr}(X),\rho)$ is quasi-unimodular with weight
\[
W(\operatorname{Tr}(X),\rho) = \myfrac[0.35em]{\E\left[\deg_G(\rho)(\#X^{-1}(\rho))^{-1} \mid (\operatorname{Tr}(X),\rho)\right]}{\E\left[\deg_G(\rho)(\#X^{-1}(\rho))^{-1}\right]}.
\]
\end{prop}

\medskip

The proof of this proposition is very similar to that of item 2 of \cref{prop:unimodularpush}, below, and is omitted.
(Note that \cref{prop:unimodular_trace} is not actually required for the proofs of \cref{thm:ends,thm:nonintersection}, but will be used in \cref{sec:further}.)

\medskip

% Let $\mu$ be an offspring distribution. 
Let us now discuss how this applies to branching random walk.
Galton-Watson trees as they are usually defined are not unimodular random rooted graphs, since the root has a special role. This can be remedied as follows. Let $\mu$ be an offspring distribution. Let $(T_1,o)$ and $(T_2,o')$ be independent Galton-Watson trees, each with offspring distribution $\mu$. Let $(T,o)$ be the rooted tree formed from $(T_1,o)$ and $(T_2,o')$ by attaching $o$ to $o'$ via a single edge. The random rooted tree $(T,o)$ is referred to as an \textbf{augmented Galton-Watson tree}, and was first considered by Lyons, Pemantle, and Peres \cite{MR1336708}. The augmented Galton-Watson tree $(T,o)$ is not  unimodular in general either, but it is quasi-unimodular with weight $\deg(o)^{-1}\E[\deg(o)^{-1}]^{-1}$ (equivalently, it is a \emph{reversible} random rooted graph). 
See \cite[Example 1.1]{AL07} for further discussion. 
We refer to a unimodular random tree $(T',o')$ whose law is obtained by biasing the law of the augmented Galton-Watson tree $(T,o)$ by $\deg(o)^{-1}$ as a \textbf{unimodular Galton-Watson tree} with offspring distribution $\mu$, and refer to the walk indexed by a unimodular Galton-Watson tree with offspring distribution $\mu$ as a \textbf{unimodular branching random walk} with offspring distribution $\mu$. Thus, in particular, \cref{prop:unimodular_trace} implies that the trace of a unimodular branching random walk on a Cayley graph is quasi-unimodular\footnote{\cite[Theorem 3.7]{MR2914859} states that this trace is unimodular, rather than quasi-unimodular; this appears to be a mistake.} with weight $\E\left[(\#X^{-1}(\rho))^{-1}\right]^{-1}\E\left[(\#X^{-1}(\rho))^{-1} \mid (\operatorname{Tr}(X),\rho)\right]$. 

\medskip

(The fact that this is the correct weight becomes intuitively clear if we think in terms of the uniformity of the root: If $f: A \to B$ is a surjective function between finite sets, and $X$ is a uniform random element of $A$, then $\P(f(X)=b) = \#f^{-1}(b)/ \# A$ for each $b\in B$. If we want to obtain a uniform measure on $B$, we should therefore bias the law of $f(X)$ by $(\#f^{-1}(f(X)))^{-1}$.)

\section{Proof of the main theorems}

\subsection{Local unimodularity}
\label{sec:localunimodularity}

We now introduce the notion of \emph{local unimodularity}. This definition plays a central role in our proofs, and we expect that it will have several further applications in the future.

% Let $\mu$ be a probability measure on $\cG_\bullet^\diamond$.  We say that $\mu$ is \textbf{locally unimodular} if $\mu(\{\rho \in A\})=1$ and
% % there exists a measurable function $W:\cG_\bullet^\diamond \to (0,\infty)$
% % such that $\E[W(G,A,\rho)]=1$ and
% \[
% \int_{\cG_\bullet^\diamond} \sum_{v\in A} F(G,A,\rho,v)\dif \mu(G,A,\rho)
% =
% \int_{\cG_\bullet^\diamond} \sum_{v\in A} F(G,A,v,\rho)\dif \mu(G,A,\rho)
% \]
% for every measurable function $F : \cG_{\bullet\bullet}^\diamond \to [0,\infty]$.  (Note that the first condition is in fact redundant, being implied by the second.)  We write $\cL(\cG_\bullet)$ for the space of locally unimodular probability measures on $\cG_\bullet^\diamond$ with the weak topology. We say that a random variable taking values in $\cG_\bullet^\diamond$ is locally unimodular if its law is locally unimodular.
% $\cG_\bullet^\diamond$
% [Add definition of $\cG_\bullet^\diamond$.]

We define $\cG_\bullet^\diamond$ 
% and $\cG_{\bullet\bullet}^\diamond$ 
to be the space of isomorphism classes of triples $(g,a,u)$, where $(g,u)$ is a rooted graph and $a$ is a distinguished set of vertices of $g$ (this notation is not standard). The local topology on $\cG_\bullet^\diamond$ is defined in an analogous way to that on $\cG_\bullet$, so that $(g,a,u)$ and $(g',a',u')$ are close in the local topology if there exists a large $r$ and an isomorphism of rooted graphs $\phi$ from the $r$-ball around $u$ in $g$ to the $r$-ball around $u'$ in $g'$ such that the intersection of $a'$ with the $r$-ball around $u'$ is equal to the image under $\phi$ of the restriction of $a$ to the $r$-ball around $u$. The doubly rooted space $\cG_{\bullet\bullet}^\diamond$ and the local topology on this space are defined analogously. It follows by a similar argument to that of \cite[Theorem 2]{CurienNotes} that $\cG_\bullet^\diamond$ and $\cG_{\bullet\bullet}^\diamond$ are Polish spaces. We write $\cT_\bullet^\diamond$ and $\cT_{\bullet\bullet}^\diamond$ for the closed subspaces of $\cG_\bullet^\diamond$ and $\cG_{\bullet\bullet}^\diamond$ in which the underlying graph is a tree.

We say that a random variable $(G,A,\rho)$ taking values in $\cG_\bullet^\diamond$ is \textbf{locally unimodular} if $\rho \in A$ almost surely and
% there exists a measurable function $W:\cG_\bullet^\diamond \to (0,\infty)$
% such that $\E[W(G,A,\rho)]=1$ and
\[
\E\left[\sum_{v\in A} F(G,A,\rho,v)\right]
=
\E\left[\sum_{v\in A} F(G,A,v,\rho)\right]
\]
for every measurable function $F : \cG_{\bullet\bullet}^\diamond \to [0,\infty]$.  (Note that the first condition is in fact redundant, being implied by the second.)  We say that a \emph{probability measure} $\mu$ on $\cG_\bullet^\diamond$ is locally unimodular if a random variable with law $\mu$ is locally unimodular. We write $\cL(\cG_\bullet^\diamond)$ for the space of locally unimodular probability measures on $\cG_\bullet^\diamond$ with the weak topology. 
% As with unimodularity, the definition of local unimodularity extends in a natural way to define the space $\cL(\cN_\bullet)$ of locally unimodular random rooted networks.
% We say in this case that $A$ is a \emph{relatively unimodular random subset of $V(G)$ with weight $W$}.

% \medskip

For example, if $(G,\rho)$ is a unimodular random rooted graph and $\omega$ is a unimodular percolation process on $G$ (i.e., $\omega$ is a random subgraph of $G$ such that $(G,\omega,\rho)$ is unimodular in an appropriate sense) and $K_\rho$ is the component of $\rho$ in $\omega$ then $(G,K_\rho,\rho)$ is locally unimodular. We stress however that locally unimodular random rooted graphs need not arise this way, and indeed that the set $A$ need not be connected. For example, if $G$ is an arbitrary connected, locally finite graph, $A$ is an arbitrary finite set of vertices of $G$, and $\rho$ is chosen uniformly at random from among the vertices of $A$ then the triple $(G,A,\rho)$ is locally unimodular. More generally, we have the intuition that $(G,A,\rho)$ is locally unimodular if and only if $\rho$ is `uniformly distributed on $A$'. (Of course, this intuitive definition does not make formal sense when $A$ is infinite.) 

% \medskip

It follows by a similar argument to \cite[Theorem 8]{CurienNotes} that $\cL(\cG_\bullet^\diamond)$ is a closed subset of the space of all probability measures on $\cG_\bullet^\diamond$ with respect to the weak topology. Thus, if $(G_n,A_n,\rho_n)$ is a sequence of locally unimodular $\cG_\bullet^\diamond$ random variables converging in distribution to $(G,A,\rho)$, then $(G,A,\rho)$ is also locally unimodular.

% \medskip

% Let us now take note of some basic properties of the space $\cL(\cG_\bullet)$.

% \begin{prop}
% $\cL(\cG_\bullet)$ is a closed subset of the space of all probability measures on $\cG_\bullet^\diamond$ with respect to the weak topology. A similar statement holds for $\cL(\cN_\bullet)$.
% \end{prop}

% \begin{proof}
% This follows very similarly to \cite[Theorem 8]{CurienNotes}.
% \end{proof}

% \begin{prop}
% A subset $K$ of $\cL(\cG_\bullet)$ is compact if and only if
% \end{prop}

% \medskip

As before, 
 it will be convenient for us to introduce the following more general notion. 
% Given a positive measurable function $W:\cG_\bullet^\diamond\to (0,\infty)$,
 We say that a random variable $(G,A,\rho)$ taking values in $\cG_\bullet^\diamond$ is \textbf{locally quasi-unimodular} if there exists a measurable function $W:\cG_\bullet^\diamond\to (0,\infty)$ such that $\E[W(G,A,\rho)]=1$ and
\[
\E\left[W(G,A,\rho)\sum_{v\in A} F(G,A,\rho,v)\right]
=
\E\left[W(G,A,\rho)\sum_{v\in A} F(G,A,v,\rho)\right]
\]
for every measurable function $F:\cG_{\bullet\bullet}^\diamond\to [0,\infty]$; in this case we say that $(G,A,\rho)$ is locally quasi-unimodular with weight $W$.
% if it is locally $W$-unimodular for some $W$.
Equivalently, $(G,A,\rho)$ is locally quasi-unimodular if and only if there exists a locally unimodular  $(G',A',\rho')$ whose law is equivalent to that of $(G,A,\rho)$ in the sense that both measures are absolutely continuous with respect to each other; the weight $W$ is the Radon-Nikodym derivative of the law of $(G',A',\rho')$ with respect to the law of $(G,A,\rho)$. (We expect that the weight $W$ has similar uniqueness properties to those discussed in \cref{remark:W_uniqueness}. We do not pursue this here.)

% \medskip

Our interest in these notions owes to the following proposition, which gives conditions under which local unimodularity can be \emph{pulled back} or \emph{pushed forward} through a unimodular tree-indexed random walk.

\begin{prop}[Local unimodularity via tree-indexed walks]
\label{prop:unimodularpush}
% Let $(G,\rho)$ be a unimodular random rooted network, let $(T,o)$ be a unimodular random rooted tree independent of $(G,\rho)$, and let $X:V(T) \to V(G)$ be a $T$-indexed random walk on $G$ with $X(o)=\rho$.
\hspace{1cm}
\begin{enumerate}
 \item \textbf{Pull-back.}  
 Let $(G,A,o)$ be a locally unimodular random rooted graph and let $(T,o)$ be an independent unimodular random rooted tree. Let $X$ be a $T$-indexed random walk on $G$ with $X(o)=\rho$. If $\E[\deg_G(\rho)] < \infty$ then $(T,X^{-1}(A),o)$ is locally quasi-unimodular with weight
% % \[
% % \int F(T,X^{-1}(A),o) \dif\lambda(T,X^{-1}(A),o) =

% % \]
%  $(T,X^{-1}(A),o)$ 
 % Let $A$ be a relatively unimodular subset of $V(G)$ with weight $W$ that is conditionally independent of $(T,o)$ and $X$ given $(G,\rho)$. If $\E\left[c(\rho)W(G,A,\rho)\right]<\infty$ then the preimage $X^{-1}(A)$ is a relatively unimodular subset of $V(T)$ with weight
 \[ W\bigl(T,X^{-1}(A),o\bigr):=\myfrac[0.4em]{\E\left[\deg_G(\rho)\mid \bigl(T,X^{-1}(A),o\bigr)\right]}{\E\left[\deg_G(\rho)\right]}. \]
 \item \textbf{Push-forward.} 
 Let $(G,o)$ be a unimodular random rooted graph and let $(T,A,o)$ be an independent locally unimodular random rooted tree. Let $X$ be a $T$-indexed random walk on $G$ with $X(o)=\rho$. If $X$ is transient almost surely and $\E[\deg_G(\rho)(\# X^{-1}(\rho))^{-1}] < \infty$ then $(G,X(A),\rho)$ is locally quasi-unimodular with weight
% % % \[
% % % \int F(T,X^{-1}(A),o) \dif\lambda(T,X^{-1}(A),o) =
% 
% % % \]
% %  $(T,X^{-1}(A),o)$ 
%  % Let $A$ be a relatively unimodular subset of $V(G)$ with weight $W$ that is conditionally independent of $(T,o)$ and $X$ given $(G,\rho)$. If $\E\left[\deg(\rho)W(G,A,\rho)\right]<\infty$ then the preimage $X^{-1}(A)$ is a relatively unimodular subset of $V(T)$ with weight
%  \[ W\bigl(T,X^{-1}(A),o\bigr):=\myfrac[0.4em]{\E\left[\deg(\rho)\mid \bigl(T,X^{-1}(A),o\bigr)\right]}{\E\left[\deg(\rho)\right]}. \]
% 
% 
%   Let $A$ be a relatively unimodular subset of $V(T)$ that is conditionally independent of $(G,\rho)$ and $X$ given $(T,o)$. If $X$ is almost surely transient and 
%  \[\E\left[\deg(\rho)W(T,A,o)|X^{-1}(\rho)|^{-1}\right]<\infty\] then the image $X(A)$ is a relatively unimodular subset of $V(G)$ with weight
 \[
W(G,X(A),\rho):= \myfrac[0.3em]{\E\left[\deg_G(\rho)(\#X^{-1}(\rho))^{-1} \mid (G,X(A),\rho)\right]}{\E\left[\deg_G(\rho)(\# X^{-1}(\rho))^{-1}\right]}.
 \]
\end{enumerate}
 % and let $A$ be the trace of an augmented $\mu$-BRW on $G$ started at $\rho$. Then $A$ is relatively unimodular.
\end{prop}

% Again, we note that this weight is equal to $1$ when $G$ is a deterministic transitive graph.

\begin{proof}[Proof of \cref{prop:unimodularpush}]

For each $(g,x) \in \cG_{\bullet}$ and $(t,u) \in \cT_\bullet$ we let $\bP_{u,x}^{t,g}$ and $\bE_{u,x}^{t,g}$ denote probabilities and expectations taken with respect to the law of a $t$-indexed random walk $X$ on $g$ started with $X(u)=x$, which we consider to be a random graph homomorphism from $t$ to $g$.
Observe that tree-indexed random walk has the following time-reversal property: If $(g,x,y) \in \cG_{\bullet\bullet}$ and $(t,u,v) \in \cT_{\bullet \bullet}$, then we have that
\begin{align}
\label{eq:timereversal1}
\deg(x) \bP_{u,x}^{t,g}\bigl(X(v)=y\bigr) &= \deg(y) \bP_{v,y}^{t,g}\bigl(X(u)=x\bigr)
% \end{align}
\intertext{and that}
% \begin{align}
\label{eq:timereversal2}
\bP_{u,x}^{t,g}\bigl(X \in \sA \mid X(v) = y\bigr) &= \bP_{y,x}^{t,g}\bigl(X \in \sA \mid X(u) = x\bigr)
\end{align}
for every event $\sA$. That is, the conditional distribution of $X$ given $\{ X(u)=x, X(v)=y\}$ is the same under the two measures $\bP_{u,x}^{t,g}$ and $\bP_{v,y}^{t,g}$.
Both statements follow immediately from the analogous statements for \emph{simple} random walk, which are classical. Indeed, $\bP_{u,x}^{t,g}\bigl(X(v)=y\bigr)$ is equal to $p_{d(u,v)}(x,y)$, so that
\eqref{eq:timereversal1} follows from the standard time-reversal identity $\deg(x) p_n(x,y) = \deg(y) p_n(y,x)$. 
To prove \eqref{eq:timereversal2}, observe that, under both measures, the conditional distribution of $X$ given $X(u)=x$ and $X(v)=y$ is given by taking the restriction of $X$ to the geodesic connecting $u$ and $v$ in $T$ to be a uniformly random path of length $d(u,v)$ from $x$ to $y$ in $G$, and then extending $X$ to the rest of $T$ in the natural Markovian fashion.
\medskip

\noindent \textbf{Proof of item 1.}
Write $\E_G$ for expectations taken with respect to  $(G,A,\rho)$ and $\E_T$ for expectations taken with respect to $(T,o)$.
Let $F: \cG_{\bullet\bullet}^\diamond \to [0,\infty]$ be measurable, and define $f: \cG_{\bullet\bullet}^\diamond \to [0,\infty]$ by
\begin{align}
f(g,a,x,y) &= \E_T\left[\sum_{v \in V(T)}\bE_{o,x}^{T,g}\left[  F\bigl(T,X^{-1}(a),o,v\bigr) \mathbbm{1}\bigl(v \in X^{-1}(y)\bigr) \right] \right].
\nonumber
% \end{align*}
\intertext{Observe that we can equivalently write $f$ as}
% \begin{align*}
f(g,a,x,y)  &=
\E_T\left[\sum_{v \in V(T)}\bE_{v,x}^{T,g}\left[  F\bigl(T,X^{-1}(a),v,o\bigr) \mathbbm{1}\bigl(o \in X^{-1}(y)\bigr) \right] \right]
\nonumber
\\
&=\frac{\deg(y)}{\deg(x)}\E_T\left[\sum_{v \in V(T)}\bE_{o,y}^{T,g}\left[  F\bigl(T,X^{-1}(a),v,o\bigr) \mathbbm{1}\bigl(v \in X^{-1}(x)\bigr) \right] \right]
\label{eq:pullback}
\end{align}
where the first equality follows from the mass-transport principle for $(T,o)$ and the second follows from the time-reversal identities \eqref{eq:timereversal1} and \eqref{eq:timereversal2}.
% This defines a measurable function $f:\cG_{\bullet\bullet}^\diamond \to [0,\infty]$. 
On the other hand, we have that
% Thus, if $(G,\rho)$ is a unimodular random rooted graph and $A$ is a relatively unimodular subset of $V(G)$ with weight function $W$ then we have that
\begin{multline*}
\E\left[  \sum_{v \in X^{-1}(A)} \deg(\rho) F\bigl(T,X^{-1}(A),o,v\bigr) \right]
\\=
 \E_G\left[  \sum_{y \in A} \deg(\rho) f(G,A,\rho,y) \right]
=
\E_G\left[   \sum_{y \in A} \deg(y)f(G,A,y,\rho) \right],\end{multline*}
where the first equality is by definition and the second follows from the mass-transport principle for $(G,A,\rho)$.
Applying \eqref{eq:pullback} we deduce that
\begin{multline*}
\E\left[  \sum_{v \in X^{-1}(A)} \deg(\rho) F\bigl(T,X^{-1}(A),o,v\bigr) \right]\\
=\E_G \left[  \sum_{y\in A}\deg(\rho)\E_T\left[\sum_{v \in V(T)}\bE_{o,\rho}^{T,g}\left[  F\bigl(T,X^{-1}(A),v,o\bigr) \mathbbm{1}\bigl(v \in X^{-1}(y)\bigr) \right] \right]\right]\\
=
\E\left[  \sum_{v \in X^{-1}(A)} \deg(\rho) F\bigl(T,X^{-1}(A),v,o\bigr) \right].
\end{multline*}
Since the measurable function $F:\cG_{\bullet\bullet}^\diamond\to[0,\infty]$ was arbitrary, this concludes the proof. 
% \[
% \E_G\left[ W(G,A,\rho) \sum_{y\in A} 
% \E_T\left[\sum_{v \in V(T)}\bE_{v,y}^{T,g}\left[  F\bigl(T,X^{-1}(a),v,o\bigr) \mathbbm{1}\bigl(o \in X^{-1}(\rho)\bigr) \right] \right]
% \right]
% \]
% \[
% =\E_G\left[ W(G,A,\rho) \sum_{y\in A} 
% \E_T\left[\sum_{v \in V(T)}\bE_{v,\rho}^{T,g}\left[  F\bigl(T,X^{-1}(a),v,o\bigr) \mathbbm{1}\bigl(o \in X^{-1}(y)\bigr) \right] \right]
% \right]
% \]

% We claim that
% \[
% \bE_{v,x}^{t,g}\left[F(t,X^{-1}(a),v,o) \mathbbm{1}(o \in X^{-1}(\rho))] = \frac{\deg(\rho)}{\deg(y)} \bE_{v,x}^{t,g}\left[F(t,X^{-1}(a),v,o) \mathbbm{1}(o \in X^{-1}(\rho))] 
% \]
% % \E_T\left[\bE_y\left[ \sum_{v \in X^{-1}(\rho)} F\bigl(T,X^{-1}(a),y,\rho\bigr) \right] \right]\right].
% % \end{align*}
% [There is a mistake here: when we apply the MTP, the walk now starts at $y$. We should need to bias by $\deg(\rho)$ here.] It follows that $X^{-1}(A)$ is a relatively unimodular subset of $V(T)$ with weight function 
% \[
% W'\bigl(T,X^{-1}(A),\rho\bigr) = \E\left[W\bigl(G,A,\rho\bigr) \mid \bigl(T,X^{-1}(A),\rho\bigr)\right].
% \]

\medskip

\noindent
\textbf{Proof of item 2.} 
Write $\E_G$ for expectations taken with respect to  $(G,\rho)$ and $\E_T$ for expectations taken with respect to $(T,A,o)$.
 Let $F:\cG_{\bullet\bullet}^\diamond\to[0,\infty]$ be measurable, and for each $(t,a,u,v) \in \cT_{\bullet\bullet}^\diamond$, define
\begin{align}
f(t,a,u,v) &= \E_G\left[ \sum_{y \in V(G)}  \deg(\rho) \bE_{u,\rho}^{t,G} \left[|X^{-1}(\rho)|^{-1} |X^{-1}(y)|^{-1} F(G,X(a),\rho,y) \mathbbm{1}(X(v)=y) \right]\right]
\nonumber
\\
&= \E_G\left[ \sum_{y \in V(G)}  \deg(y)\bE_{u,y}^{t,G} \left[ |X^{-1}(\rho)|^{-1} |X^{-1}(y)|^{-1}F(G,X(a),y,\rho) \mathbbm{1}(X(v)=\rho) \right]\right]
\nonumber
\\
&= \E_G\left[  \sum_{y \in V(G)}\deg(\rho) \bE_{v,\rho}^{t,G} \left[ |X^{-1}(\rho)|^{-1} |X^{-1}(y)|^{-1}F(G,X(a),y,\rho) \mathbbm{1}(X(u)=y) \right]\right],
\label{eq:pushforward}
\end{align}
where, as before, the first equality follows from the mass-transport principle for $(G,\rho)$ and the second inequality follows from the time-reversal identities \eqref{eq:timereversal1} and \eqref{eq:timereversal2}.
Taking expectations over $(T,A,o)$, we deduce that
\begin{align*}
\E\left[ \deg(\rho)|X^{-1}(\rho)| \sum_{y\in X(A)} F(G,X(A),\rho,y)\right]&
\\&\hspace{-3.7cm}=\E\left[\sum_{y\in V(G)} \deg(\rho) \sum_{v\in A}|X^{-1}(\rho)||X^{-1}(y)|^{-1}   F(G,X(A),\rho,y) \mathbbm{1}(X(v)=y)\right]
\\&\hspace{-3.7cm}= \E_T\left[ \sum_{v\in A}f(T,A,o,v)\right]= \E_T\left[ \sum_{v\in A}f(T,A,v,o)\right],
\end{align*}
where the first and second equalities are by definition and the third is by the mass-transport principle for $(T,A,o)$. Applying \eqref{eq:pushforward} we deduce that
\begin{multline*}
\E\left[  \deg(\rho)|X^{-1}(\rho)| \sum_{y\in X(A)} F(G,X(A),\rho,y)\right] \\
= \E_T\left[  \sum_{v\in A} \E_G \left[\sum_{y \in V(G)}  \deg(\rho) \bE_{o,\rho}^{T,G} \left[ |X^{-1}(\rho)|^{-1}|X^{-1}(y)|^{-1}F(G,X(a),y,\rho) \mathbbm{1}(X(v)=y) \right]\right]\right]\\
=
\E\left[  \deg(\rho)|X^{-1}(\rho)| \sum_{y\in X(A)} F(G,X(A),y,\rho)\right].
\end{multline*}
The claim follows since the measurable function $F:\cG_{\bullet\bullet}^\diamond\to[0,\infty]$ was arbitrary.
\end{proof}

Note that the weight that arises when pulling back is identically equal to $1$ when $G$ is a deterministic transitive graph.
Moreover, pushing forward $A=V(T)$, it follows that if $\E[\deg(\rho)\bigl(\#X^{-1}(\rho)\bigr)^{-1}]<\infty$ then $(G,X(V(T)),\rho)$ is locally quasi-unimodular with weight
\[
W\bigl(G,X(V(T)),\rho\bigr) = \myfrac[0.375em]{ \E\left[\deg_G(\rho)\bigl(\#X^{-1}(\rho)\bigr)^{-1} \mid \bigl(G,X(V(T)),\rho\bigr)\right]}{\E\left[\deg_G(\rho)(\#X^{-1}(\rho))^{-1}\right]}.
\]
This is very closely related to \cref{prop:unimodular_trace}.
% (If $G$ is a deterministic transitive graph then this weight is $1$.)
Pulling this set back along a second tree-indexed walk, we therefore deduce the following immediate corollary.

% For our purposes, we will be primarily interested in the following immediate corollary.

\begin{corollary}
\label{cor:unimodularintersection}
Let $G$ be a connected, locally finite, unimodular transitive graph and let $\rho$ be a vertex of $G$. For each $i \in \{1,2\}$, let $(T_i,o_i)$ be a unimodular random rooted tree and let $X_i$ be a $T_i$-indexed random walk on $G$ with  $X_i(o_i)=\rho$, where we take the random variables $((T_1,o_1),X_1)$ and $((T_2,o_2),X_2)$ to be independent. Let $I=X_1^{-1}(X_2(V(T_2))) \subseteq V(T_1)$.
% let $X_1:V(T_1)\to V(G)$ and $X_2 : V(T_2) \to V(G)$ be random walks in $G$ indexed by $T_1$ and $T_2$ respectively, with $X_1(o_1)=X_2(o_2)=\rho$ and with $X_i$ conditionally independent of $X_{3-i}$ and $(T_{3-i},o_{3-i})$ given $(G,\rho)$ and $(T_i,o_i)$. Let $I:=X_1^{-1}(X_2(V(T_2)))$.
 If $X_2$ is almost surely transient, then the random triple
$(T_1,I,o_1)$ is locally quasi-unimodular with weight
\[
W(T_1,I,o_1) = \myfrac[0.4em]{\E\left[ \bigl(\#X_2^{-1}(\rho)\bigr)^{-1} \mid  \bigl(T_1,I,o_1\bigr)\right]}{\E\left[ \bigl(\#X_2^{-1}(\rho)\bigr)^{-1}\right]}.
\]
\end{corollary}

Note that this corollary has a straightforward extension to the case that $(G,\rho)$ is a unimodular random rooted graph or network. (Indeed, one can even consider the case that $G$ carries two different network structures, one for each walk, in a jointly unimodular fashion.)

\subsection{Ends in locally unimodular random trees via the Magic Lemma}

Recall that an infinite graph $G$ is said to be \textbf{$k$-ended} (or that $G$ \textbf{has $k$ ends}) if deleting a finite set of vertices from $G$ results in a maximum of $k$ infinite connected components. 
It is a well-known fact that a Benjamini-Schramm limit of finite trees (i.e., a distributional limit of finite trees each rooted at a uniform random vertex) is either finite or has at most two ends.
 % More generally, we have the following fact:
% 
% \begin{theorem}
% \label{thm:standardmagic}
% Let $(T,o)$ be a unimodular random rooted tree. If $\E[\deg_G(\rho)]\leq 2$, then $T$ is either finite or has at most two ends almost surely.
% \end{theorem}
% 
There are several ways to prove this (see e.g.\ \cite[Theorem 13]{CurienNotes}), and several far-reaching generalizations of this fact can be found in \cite{BeSc,AL07,unimodular2}. 
% (The fact that Benjamini-Schramm limits of finite trees have expected degree at most $2$ follows from Fatou's Lemma.)

 Our next result shows that this fact also has a \emph{local} version, from which we will deduce \cref{thm:ends,thm:nonintersection} in the next subsection.
Given a graph $G$ and an infinite set of vertices $A$ in $G$, we say that $A$ is $k$-ended if deleting a finite set of vertices from $G$ results in a maximum of $k$ connected components that have infinite intersection with $A$. (In particular, if $T$ is a tree, then an infinite set of vertices $A$ in $T$ is $k$-ended if and only if it accumulates to exactly $k$ ends of $T$.) 

\begin{theorem}
\label{thm:magic}
Let $((T_n,A_n,o_n))_{n\geq 1}$ be a sequence of locally unimodular random rooted trees converging in distribution\footnote{This means that the law $\mu_n$ of $(T_n,A_n,o_n)$ converges to the law $\mu$ of $(T,A,o)$ in the weak topology on the space of probability measures on $\cG_\bullet^\diamond$ associated to the local topology on $\cG_\bullet^\diamond$. } to some random variable $(T,A,o)$ as $n\to \infty$. 
If $A_n$ is finite almost surely for every $n \geq 1$, then $A$ is either finite, one-ended, or two-ended almost surely.
\end{theorem}

% \begin{remark}
% This theorem is very closely related to the \emph{Magic Lemma} of Benjamini and Schramm \cite[Lemma 2.3]{BeSc}; see in particular the hyperbolic formulation of the Magic Lemma given in \cite[Section 4]{1804.10191}, which implies the bounded degree case of \cref{thm:magic}. The assumption that the $(T_n,A_n,o_n)$ converge in distribution can be thought of as a tightness condition that replaces the role played by the bounded degree assumption in that result.
% \end{remark}

We will deduce \cref{thm:magic} as a corollary of \cref{thm:magicmagic}, below. This theorem is a version of the \emph{Magic Lemma} of Benjamini and Schramm  \cite[Lemma 2.3]{BeSc}, see also \cite[Section 5.2]{AsafBook}. Indeed, while the usual statement of the Magic Lemma concerns sets of points in $\R^d$, its proof is powered by a more fundamental fact about trees,  which is implicit in the original proof (see in particular \cite[Claim 5.5]{AsafBook}) and is essentially equivalent to \cref{thm:magicmagic}. We include a full proof for clarity, and since the statement we give is slightly different.
We remark that the Magic Lemma has found diverse applications to several different problems in probability \cite{BeSc,GN12,GiRo13,1804.10191}, and useful generalizations of the Magic Lemma to doubling metric spaces \cite{MR3266996} and to Gromov hyperbolic spaces \cite{1804.10191} have also been found.

% We now introduce the definitions needed to state the theorem.
Let $T$ be a locally finite tree and let $A$ be a finite set of vertices of $T$. 
For each pair of distinct vertices $u,v$ in $T$, let $A_{u,v}$ be the set of vertices $a\in A \setminus \{v\}$ such that the unique simple path from $u$ to $a$ in $T$ passes through $v$. 
% To avoid trivialities, we also set $A_{u,u}=\emptyset$ for every vertex $u$.
We say that a vertex $u$ of $T$ is \textbf{$(k,r)$-branching} for $A$ if $|A|-|A_{u,v} \cup A_{u,w}| \geq k$ for every pair of vertices $v,w$ with distance exactly $r$ from $u$.
% in the ball of radius $r$ around $u$ in $T$.

\begin{thm}[Magic lemma for trees]
\label{thm:magicmagic}
Let $T$ be a locally finite tree and let $A$ be a finite set of vertices of $T$. Then for each $k,r \geq 1$, there are at most $r(2|A|-k)/k$ vertices of $T$ that are $(k,r)$-branching for $A$.
\end{thm}

\begin{proof}
By attaching an infinite path to $T$ if necessary, we may assume without loss of generality that $T$ is infinite. 
We may then pick an orientation of $T$ so that every vertex $v$ of $T$ has exactly one distinguished neighbour, which we call the parent of $v$ and denote by $\sigma(v)$. This leads to a decomposition $(L_n)_{n \in \Z}$ of $T$ into layers, unique up to a shift of index, such that the parent of every vertex in $L_n$ lies in $L_{n-1}$ for every $n \in \Z$. These levels are sometimes known as \emph{horocycles}, see e.g. \cite[Section II.12.C]{Woess}. (It may be that $L_n =\emptyset$ for every $n$ larger than some $n_0$, but this possibility will not cause us any problems.)
We denote by $\sigma^r$ the $r$-fold iteration of $\sigma$, so that if $v\in L_{n}$ then $\sigma^r(v) \in L_{n-r}$.
We call $u$ a \textbf{descendant} of $v$, and call $v$ an \textbf{ancestor} of $u$, if $v=\sigma^r(u)$ for some $r\geq 0$. For each vertex $v$ of $T$, we let $A_v$ be the set of vertices in $A \setminus \{v\}$  that are descendants of $v$.

% Let $\rho$ be an arbitrary vertex of $T$. For each $n\geq 0$, let $L_n$ be the set of vertices with distance exactly $n$ from $\rho$. 
% For each vertex $v$ of $T$ other than $\rho$, we call a vertex $u$ a descendant of $v$, and call $v$ an ancestor of $u$, if the unique simple path from $\rho$ to $u$ passes through $v$. Every vertex is considered to be a descendant of $\rho$. For each vertex $v$ of $T$, let $A_v$ be the set of vertices in $A$  that are descendants of $v$.
% For each $n\geq 1$ and each vertex $v\in L_n$, let $\sigma(v)$ denote the parent of $v$, that is, the unique neighbour of $v$ in $L_{n-1}$, and let $\sigma^r : \bigcup_{n\geq r} L_{n} \to V$ denote the $r$-fold iteration of $\sigma$.  

We say that a vertex $v$ is $(k,r)$\textbf{-supported} if $|A_{u}|-|A_{w}| \geq k$ for every $w$ with $\sigma^r(w)=v$. Observe that 
for every vertex $u$ and every $w$ with $\sigma^r(w)=u$, we have that $A_w = A_{u,w} \subseteq A_u$ and that $A_u \subseteq A \setminus A_{u,\sigma^r(u)}$, so that 
$|A_u|-|A_w| \geq |A| - |A_{u,\sigma^r(u)} \cup A_{u,w}|$. Thus, every $(k,r)$-branching vertex is $(k,r)$-supported, and it suffices to prove that there exist at most $r(2|A|-k)/k$ vertices that are $(k,r)$-supported. We may assume that $|A|\geq k$, since otherwise there cannot be any $(k,r)$-supported vertices and the claim holds vacuously. 

We begin with the case $r=1$. We follow closely the proof of \cite[Claim 5.5]{AsafBook}. Let $V$ be the vertex set of $T$, and  let $B$ be the set of $(k,1)$-supported points.
Define a function $f: V^2 \to \R$ by
\[
f(u,v)  = 
\begin{cases}
\phantom{-}|A_u| \wedge \frac{k}{2} & v = \sigma(u)\\
-|A_v|\wedge \frac{k}{2} & u = \sigma(v)\\
\phantom{-}0 & \text{otherwise}.
\end{cases}
\]
This function is antisymmetric in the sense that $f(u,v)=-f(v,u)$  for every $u,v \in V$. We observe that
\[
0 \leq f(u,\sigma(u)) = \frac{k}{2}\wedge \sum_{v: \sigma(v)=u}\left[\mathbbm{1}(v \in A)+f(v,u)\right]  \leq \sum_{v: \sigma(v)=u}\left[\mathbbm{1}(v \in A)+f(v,u)\right] 
\]
for every $u \in V$, as can be verified by splitting into two cases according to whether $u$ has a child $v$ with $|A_v|\geq k/2$ or not. Moreover, if $u$ is $(k,1)$-supported then
\[
 f(u,\sigma(u))   \leq \sum_{v: \sigma(v)=u}\left[\mathbbm{1}(v \in A)+f(v,u)\right]  - \frac{k}{2},
\]
where the inequality may be verified by splitting into three cases according to whether $u$ has zero, one, or more than one child $v$ with $|A_v| \geq k/2$.

Let $S$ be the finite set spanned by the union of the geodesics between pairs of points in $A$. Observe that $B \cup A \subseteq S$ and that if $v\notin S$ then $A_v \in \{A,\emptyset\}$. Note also that there is a unique vertex $\rho \in S$ such that every vertex of $S$ is descended from $\rho$, and this vertex $\rho$ satisfies $A = A_{\sigma(\rho)}$. Let $S' = S \cup \{\sigma(\rho)\}$.
% Finally, observe that $f(u,\sigma(u))=\sum_{v:\sigma(v):u} f(v,u) \in \{0,k/2\}$ for every $u$ that does not lie in the finite set $S$ spanned by the union of the geodesics between points of $A$. 
We may sum the above estimates to obtain that
% 
% \[
% \sum_{u\in S}\left[f(u,\sigma(u))+\sum_{v:\sigma(v)=u}f(u,v)-\mathbbm{1}(v\in A)\right] \leq 0
% \]
% 
% 
\begin{multline*}
|A|-\frac{k}{2}|B| + \sum_{u \in S'} \left[\sum_{v: \sigma(v)=u}f(v,u) - f(u,\sigma(u))\right]
\\=
 \sum_{u \in S'} \left[\sum_{v: \sigma(v)=u}\left[\mathbbm{1}(v \in A)+f(v,u)\right] - f(u,\sigma(u)) - \frac{k}{2}\mathbbm{1}\bigl(u\in B)\right] \geq 0.
\end{multline*}
On the other hand, using the antisymmetry property of $f$ and rearranging we obtain that
\begin{align*}
\sum_{u \in S'} \left[\sum_{v: \sigma(v)=u}f(v,u) - f(u,\sigma(u))\right] &= \sum_{v \notin S', \sigma(v) \in S'} f(v,\sigma(v)) + \sum_{u,v\in S'} f(u,v) - f(\sigma(\rho),\sigma^2(\rho))\\
&=-f(\sigma(\rho),\sigma^2(\rho)) = - \frac{k}{2}, 
\end{align*}
so that
$\frac{k}{2} |B| \leq |A|-\frac{k}{2}$
as claimed.

Now let $r\geq 2$. We will deduce the bound in this case from the $r=1$ bound by constructing an auxiliary tree corresponding to each residue class mod $r$. For each $1 \leq m \leq r$, let 
$R_m = \bigcup_{n\in \Z} L_{nr+m}$ and let
$T_m$ be the tree constructed from $T$ by connecting each vertex in a level of the form $L_{nr + m}$ to all of its descendants in $\bigcup_{\ell=1}^r L_{nr+m+\ell}$. Thus, $T_m$ has the same vertex set as $T$, and every vertex not in $R_m$ is a leaf in $T_m$. Observe that if a vertex $v\in R_m$ is $(k,r)$-supported in $T$ then it is $(k,1)$-supported in $T_m$. For each $1\leq m \leq r$ we know that there are at most $(2|A|-k)/k$ such vertices, and the claim follows by summing over $m$.
\end{proof}

\begin{proof}[Proof of \cref{thm:magic}]
Let $(T,A,o)$ be locally unimodular and suppose that $A$ is almost surely finite. Let $k,r \geq 1$ and let $B_{k,r}$ be the set of vertices of $T$ that are $(k,r)$-branching for $A$. Considering the function $F:\cT_{\bullet\bullet}^\diamond\to [0,\infty]$ defined by $F(g,a,u,v) = \mathbbm{1}(v$ is $(k,r)$-branching for $a)/|a|$, and applying the mass-transport principle, we obtain that
\[
\E\left[|B_{k,r} \cap A|/|A|\right] = 
\E\left[ \sum_{v\in A}F(G,A,\rho,v)\right] = \E\left[ \sum_{v\in A}F(G,A,v,\rho)\right] = \P(o \in B_{k,r}).
\]
Applying \cref{thm:magicmagic} to bound the left hand side, we obtain that
\begin{equation}
\label{eq:magic_unimod_1}
\P(o \in B_{k,r}) \leq \frac{2r}{k}
\end{equation}
for every $k,r\geq 1$ and every locally unimodular triple $(T,A,o)$ such that $A$ is almost surely finite.

Now observe that for each $k,r\geq 1$, the set of $(t,a,u) \in \cT_\bullet^\diamond$ such that $u$ is $(k,r)$-branching for $a$ is open with respect to the local topology on $\cT_\bullet^\diamond$. It follows by the portmanteau theorem that the map $\mu \mapsto \mu(\{(t,a,u): u \text{ is $(k,r)$-branching for $a$}\})$ is weakly lower semi-continuous on the space of probability measures on $\cT_\bullet^\diamond$. We deduce that if $((T_n,A_n,o_n))_{n\geq 1}$ and $(T,A,o)$ are as in the statement of the theorem then
\begin{equation}
\label{eq:magic_unimod_2}
\P\Bigl(\text{$o$ is $(k,r)$-branching for $A$}\Bigr) \leq \lim_{n\to\infty} \P\Bigl(\text{$o_n$ is $(k,r)$-branching for $A_n$}\Bigr) \leq \frac{2r}{k}
\end{equation}
for every $r,k \geq 1$. This is a quantitative refinement of the statement of the theorem: If $A$ is infinite with  more than two ends then there exists a vertex $v$ of $T$ whose removal disconnects $T$ into at least three connected components that have infinite intersection with $A$. If there is such a vertex within distance $r$ of $o$, then $o$ is $(k,r)$-branching for every $k\geq 1$. The estimate \eqref{eq:magic_unimod_2} implies that this event has probability zero for every $r \geq 1$, and the claim follows.
\end{proof}

\subsection{Completing the proof}

We now deduce \cref{thm:ends,thm:nonintersection} from \cref{thm:magic}. We begin with the following simple lemma.

\begin{lemma}
\label{lem:subcritical}
Let $G$ be a transitive nonamenable graph with spectral radius $\|P\|<1$, and let $\mu_1,\mu_2$ be offspring distributions with $\overline{\mu_1},\overline{\mu_2} \leq \|P\|^{-1}$, and suppose that this inequality is strict for at least one of $i=1,2$. Let $x,y$ be vertices of $G$. Then an independent $\mu_1$-BRW started at $x$ and $\mu_2$-BRW  started at $y$ intersect at most finitely often almost surely.
\end{lemma}

\begin{proof}[Proof of \cref{lem:subcritical}]
For $i=1,2$, let $T_i$ be a $\mu_i$-Galton-Watson tree with root $o_i$ and let $X_i$ be a random walk on $G$ indexed by $T_i$, started at $x$ when $i=1$ and $y$ when $i=2$, where the pair $(T_1,X_1)$ is independent of $(T_2,X_2)$. Let $V_i$ be the vertex set of $T_i$. The expected number of vertices of $T_i$ with distance exactly $n$ from $o_i$ is $\overline{\mu_i}^n$, and we can compute that
\begin{align*}\E\left[\#\{(u,v) \in V_1 \times V_2 : X_1(u)=X_2(v)\} \right]
&= 
\sum_{z \in V(G)} \sum_{n,m 
\geq 0} \overline{\mu_1}^n p_n(x,z) \overline{\mu_2}^m p_m(y,z)\\
&= \sum_{z \in V(G)} \sum_{n,m 
\geq 0} \overline{\mu_1}^n p_n(x,z) \overline{\mu_2}^m p_m(z,y)\\
&= \sum_{n,m 
\geq 0}  \overline{\mu_1}^n\overline{\mu_2}^m p_{n+m}(x,y).
% &= C\sum_{n,m \geq 0 }p_n \rangle 
% \\&\leq C \sum_{n,m \geq 0 }\overline{\mu_1}^n \|P\|^n \overline{\mu_2}^m  \|P\|^m < \infty,
\end{align*}
Since $\overline{\mu_1},\overline{\mu_2} \leq \|P\|^{-1}$ and  this inequality is strict for at least one of $i=1,2$, it follows by an elementary calculation that there exists a constant $C$ such that
\[
\E\left[\#\{(u,v) \in V_1 \times V_2 : X_1(u)=X_2(v)\} \right]
\leq C \sum_{n \geq 0} \|P\|^{-n}p_n(x,y).
\]
The right-hand side is finite by \cite[Theorem 7.8]{Woess}, concluding the proof.  (Note that we do not need to invoke this theorem if we have both strict inequalities $\overline{\mu_1},\overline{\mu_2} < \|P\|^{-1}$, and in this case the claim holds for any bounded degree nonamenable graph.)
\end{proof}

% \begin{lemma}
% \label{lem:Markov}
% Let $G$ be a graph, and let $A$ be a set of vertices of $G$. Let $\mu$ be a non-trivial offspring distribution, let $T$ be a Galton-Watson tree with offspring distribution $\mu$, and let $X$ be a $T$-indexed random walk on $G$ started at some vertex $x$. Then the event that $X^{-1}(A)$ is an infinite subset of $T$ with finitely many ends has probability zero.
% \end{lemma}

% \begin{proof}[Proof of \cref{lem:Markov}]
% For each vertex $y$ of $G$, let $F(y)=\bE_y\left[X^{-1}(A)=\infty\right]$ be the probability that a $\mu$-branching random walk started at $y$ has infinite intersection with $A$. Let $\cF_n$ be the $\sigma$-algebra generated by the first $n$ generations of the Galton-Watson tree $T$ restriction of $X$ to these generations. Then we have that
% \[
% \bE_x\left[X^{-1}(A)=\infty \mid \cF_n \right] = 1-\prod_{v \in \partial T_n} \left(1-F(X(v))\right)
% \]
% \end{proof}

Given an offspring distribution $\mu$ and $p\in [0,1]$, let $\mu^p$ be the offspring distribution 
defined by
\[
\mu^p(k) = \sum_{n\geq k} \binom{n}{k}p^k(1-p)^{n-k}\mu(k),
\]
so that $\overline{\mu^p}=p \overline{\mu}$ and $\mu^p$ converges weakly to $\mu=\mu^1$ as $p\uparrow 1$.

\begin{proof}[Proof of \cref{thm:nonintersection}]
First, observe that the claim is clearly equivalent to the corresponding claim concerning \emph{unimodular} branching random walks. 
 Moreover, 
it suffices to consider the case that $x=y=\rho$, where $\rho$ is some fixed root vertex of $G$. Indeed, if there exists \emph{some} choice of starting vertices $x$ and $y$ so that the two walks intersect infinitely often with positive probability, then \emph{any} choice of starting vertices must have this property, since there exist times $n$ and $m$ such that with positive probability the first walk has at least one particle at $x$ at time $n$ and the second walk has at least one particle at $y$ at time $m$, and on this event we clearly have a positive conditional probability of having infinitely many intersections.
We may also assume that the offspring distributions 
 $\mu_1,\mu_2$ have $\overline{\mu_1},\overline{\mu_2}= \|P\|^{-1}>1$, since otherwise the claim follows from \cref{lem:subcritical}. 
 In particular, this implies that both $\mu_1$ and $\mu_2$ are non-trivial. 
% Finally, 

 For each $i\in\{1,2\}$ let $(T_i,o_i)$ be a unimodular Galton-Watson tree with offspring distribution $\mu_i$,  let $X_i$ be a $T_i$ indexed random walk on $G$ with $X_i(o_i)=\rho$, and let $U_i=(U_i(e))_{e\in E(T_i)}$ be a collection of i.i.d.\ uniform $[0,1]$ random variables indexed by the edge set of $T_i$. We take $X_i$ and $U_i$ to be conditionally independent given $T_i$ for each $i=1,2$, and take the two random variables $((T_1,o_1),X_1,U_1)$ and $((T_2,o_2),X_2,U_2)$ to be independent of each other.
We have by the results of \cite{MR2284404,MR2426846} that $X_1$ and $X_2$ are both transient almost surely. Let $I = X_1^{-1}(X_2(V(T_2)))$. We wish to show that $I$ is finite almost surely.

For each $i\in \{1,2\}$ and $p\in [0,1]$, let $T_i^p$ be the component of $o_i$ in the subgraph of $T_i$ spanned by the edges of $T_i$ with $U_i(e) \leq p$. Let $X_i^p$ be the restriction of $X_i$ to $T_i^p$. Then $(T_i^p,o_i)$ is a unimodular random tree, and $X_i^p$ is distributed as a $T_i^p$-indexed random walk on $G$. Observe that we can alternatively sample a random variable whose law is equivalent (i.e., mutually absolutely continuous) to that of $(T_i^p,o_i)$ 
by taking two independent Galton-Watson trees with law $\mu_i^p$, attaching these trees by a single edge between their roots, and then deciding whether to delete or retain this additional edge with probability $p$, independently of everything else. It follows from this observation together with \cref{lem:subcritical} that the set  
$I^p := (X_1^p)^{-1}(X_2^p(V(T_2^p)))$ is almost surely finite when $p<1$.

By \cref{cor:unimodularintersection}, for each $p\in [0,1]$ the random triple
$(T_1^p,I^p,o_1)$ is locally quasi-unimodular with weight
\[
W_p(T_1^p,I^p,o_1) = \myfrac[0.35em]{\E\left[\left(\#(X_2^p)^{-1}(\rho)\right)^{-1} \mid (T_1^p,I^p,o_1)\right]}{\E\left[\left(\#(X_2^p)^{-1}(\rho)\right)^{-1}\right]}.
\]
For each $p\in [0,1]$ let $W_p'$ be the random variable 
\[
W_p':=\myfrac[0.3em]{\left(\#(X_2^p)^{-1}(\rho)\right)^{-1}}{\E\left[\left(\#(X_2^p)^{-1}(\rho)\right)^{-1}\right]^{-1}},
\]
so that $W_p(T_1^p,I^p,o_1) = \E[W'_p \mid (T_1^p,I^p,o_1)]$. 
Since $X_2=X_2^1$ is transient, the expectation in the denominator is bounded away from $0$. Since we also trivially have that $\left(\#(X_2^p)^{-1}(\rho)\right)^{-1} \leq 1$, it follows that the 
random variables $W'_p$ are all bounded by the finite constant $1/\E\bigl[\left(\#(X_2)^{-1}(\rho)\right)^{-1}\bigr]$. Moreover, we clearly have that $W'_p \to W_1'$ almost surely as $p \uparrow 1$. For each $p\in [0,1]$, let $\nu_p$ be the law of $(T_1^p,I^p,o_1)$ and let $\nu_p'$ be the locally unimodular probability measure given by biasing $\nu_p$ by $W_p$. We clearly have that $\nu_p$ converges weakly to $\nu_1$ as $p\uparrow 1$, and we claim that $\nu_p'$ converges weakly to $\nu_1'$ as $p\uparrow 1$ also. Indeed, if $F:\cG_\bullet^\diamond \to \R$ is a bounded continuous function then we trivially have that $F(T_1^p,I^p,o_1)$ converges almost surely to $F(T_1,I,o_1)$ as $p \uparrow 1$, and it follows by bounded convergence that
\begin{align*}
\lim_{n\to\infty}\E\left[W_p(T_1^p,I^p,o_1)F(T_1^p,I^p,o_1)\right] &= \lim_{n\to\infty}\E\left[W_p'F(T_1^p,I^p,o_1)\right] \\&= \E\left[W_1'F(T_1,I,o_1)\right] = \E\left[W_1(T_1,I,o_1)F(T_1,I,o_1)\right].
\end{align*}
 Since $F$ was arbitrary, this establishes the desired weak convergence.
Since the sets $I^p$ are almost surely finite for every $0\leq p < 1$, it follows from \cref{thm:magic} that $I=I^1$ is either finite, one-ended or two-ended almost surely.

Suppose for contradiction that $I$ is infinite with positive probability. Since $\mu_1$ and $\mu_2$ are non-trivial, there exists $n$ such that, with positive probability, $o_1$ and $o_2$ both have at exactly three descendants belonging to $X^{-1}(\rho)$ in level $n$.
Condition on the $\sigma$-algebra $\cF$ generated by the first $n$ generations of each tree and the restriction of $X$ to these generations, and suppose that this event holds. Denote the three descendants in each tree by $o_{i,1},o_{i,2},o_{i,3}$ (the choice of enumeration is not important), let $T_{i,j}$ be the subtree of $T_i$ spanned by $o_i$ and its descendants, and let $X_{i,j}$ be the restriction of $X_i$. Then $T_{i,j}$ is conditionally distributed as a Galton-Watson tree with offspring distribution $\mu_i$, and $X_{i,j}$ is a $T_{i,j}$-indexed walk on $G$ started with $X_{i,j}(o_{i,j})=\rho$. Moreover, the random variables $((T_{i,j},o_{i,j}),X_{i,j})$ are all conditionally independent of each other given $\cF$, and our assumption implies that $X_{1,j}^{-1}(X_{2,j}(V(T_{2,j})) = \infty$ with positive conditional probability for each $1 \leq j \leq 3$. It follows by independence that $X_{1,j}^{-1}(X_{2,j}(V(T_{2,j})) = \infty$ for every $1 \leq j \leq 3$ with positive probability, and hence that $I$ has at least three ends with positive probability, a contradiction.
\end{proof}

% \begin{remark}
% Using the estimate \eqref{eq:magic_unimod_2}, one can make the above argument quantitative and prove that there exists a constant $C$ such that
% \[
% \E\left[\bigl(\#X_2^{-1}(\rho)\bigr)^{-1}\mathbbm{1}\left(\#X_1^{-1}\left(X_2\left(V(T_2)\right)\right) \geq n \right) \right] \leq C n^{-1/3}.
% \]
% With  more work we expect that this can be improved to a bound of the form 
% \[
% \E\left[\bigl(\#X_2^{-1}(\rho)\bigr)^{-1}\mathbbm{1}\left(\#X_1^{-1}\left(X_2\left(V(T_2)\right)\right) \geq n \right) \right] \leq C n^{-1/2}
% \]
% which we expect to be optimal.
% \end{remark}

% \section{New proof}

\begin{remark}
The last part of the proof of \cref{thm:nonintersection} can be generalized as follows: Suppose that $G$ is a graph, $\mu$ is a non-trivial offspring distribution, $T$ is a Galton-Watson tree with offspring distribution $\mu$, and $X$ is a $T$-indexed random walk in $G$. Let $A$ be a set of vertices in $G$. Then the event $\{X^{-1}(A)$ is infinite and has finitely many ends$\}$ has probability zero.
\end{remark}

It remains to deduce \cref{thm:ends} from \cref{thm:nonintersection}; this is very straightforward. We also prove the following slight variation on the same result.

\begin{theorem}
\label{thm:endsunimod}
Let $G$ be a unimodular transitive graph.
 % and let $P$ be a symmetric, $\Aut(G)$-invariant transition matrix on $G$ with $\rho(P)<1$. 
 Let $\mu$ be an offspring distribution with $1<\overline{\mu} \leq \|P\|^{-1}$. Then the trace of a \emph{unimodular} branching random walk on $G$ with offspring distribution $\mu$ is infinitely ended and has no isolated ends almost surely on the event that it survives forever.
\end{theorem}

\begin{proof}[Proof of \cref{thm:ends,thm:endsunimod}]
% We prove that the trace has infinitely many ends almost surely on the event that it is infinite; the fact that it does not have any isolated ends then follows from \cref{prop:unimodular_trace} and \cite[Proposition 6.10]{AL07}. (A small argument is needed to pass between statements about  branching random walks and statements about \emph{unimodular}  branching random walks, but this is straightforward and we omit the details.)
We begin by proving that the trace of a branching random walk is infinitely-ended on the event that it survives forever.
Let $(T,o)$ be a Galton-Watson tree with offspring distribution $\mu$, and let $X$ be a $T$-indexed random walk in $G$ with $T(o)=\rho$. Let $\cF_n$ be the $\sigma$-algebra generated by the first $n$ generations of $T$ and the restriction  of $X$ to these generations. Let the vertices of $T$ in generation $n$ be enumerated $v_{n,1},\ldots,v_{n,N_n}$, and let $M_n$ be the number of vertices in generation $n$ that have infinitely many descendants. Let $W_n$ be the image of the first $n$ generations of $T$ under $X$ and let $A_{n,i}$ be the image under $X$ of the offspring of $v_{n,i}$. \cref{thm:nonintersection} implies that $|A_{n,i} \cap A_{n,j}| < \infty$ for every $1 \leq i < j \leq N_n$. Let $K_n = W_n \cup \bigcup_{1 \leq i < j \leq N_n} A_{n,i} \cap A_{n,j}$. 
Then $K_n$ is finite and deleting $K_n$ from the trace of $X$ results in at least $M_n$ infinite connected components. On the other hand, standard results in the theory of branching processes imply that $M_n \to \infty$ almost surely on the event that $T$ is infinite, concluding the proof. A similar proof establishes that the trace of a \emph{unimodular} branching random walk is infinitely-ended almost surely on the event that it survives forever.

Applying \cref{prop:unimodular_trace} and \cite[Proposition 6.10]{AL07}, we deduce that the trace of a unimodular branching random walk has continuum ends and no isolated end almost surely on the event that it survives forever. The fact that the same claim holds for the usual branching random walk trace follows by a further application of \cref{thm:nonintersection}. This deduction will use the notion of the \emph{space of ends} of a tree as a topological space, see \cite[Section 21]{Woess} for a definition. 
Let $(T_1,o)$ and $(T_2,o')$ be independent Galton-Watson trees with offspring distribution $\mu$, and let $(T,o)$ be the augmented Galton-Watson tree formed by attaching $(T_1,o)$ and $(T_2,o')$ by a single edge connecting $o$ to $o'$.
Let $X$ be a $T$-indexed random walk with $X(o)=\rho$, and let $X_1$ and $X_2$ be the restrictions of $X$ to $T_1$ and $T_2$ respectively, so that $\operatorname{Tr}(X)$ has continuum many ends and no isolated ends almost surely on the event that it is infinite. \cref{thm:nonintersection} is easily seen to imply that the space of ends of $\operatorname{Tr}(X)$ is equal to the disjoint union of the spaces of ends of $\operatorname{Tr}(X_1)$ and $\operatorname{Tr}(X_2)$, and it follows that $\operatorname{Tr}(X_1)$ has continuum many ends and no isolated end almost surely on the event that $T_1$ is infinite, as desired.
\end{proof}

% \subsection{Proof of Theorem \ref{thm:percolation}}

\section{Further results}

\label{sec:further}

We now discuss how several properties of the unimodular random tree $(T,o)$ are inherited by the quasi-unimodular random rooted graph $(\operatorname{Tr}(X),\rho)$. Since the material is tangential to the main topic of the paper, we will be a little brief and refer the reader to \cite{AL07,unimodular2,CurienNotes} for more detailed treatments of the associated definitions.

\medskip
\noindent
\textbf{Hyperfiniteness.} Roughly speaking, a unimodular random rooted graph is said to be \textbf{hyperfinite} if it can be exhausted by finite subgraphs of itself in a jointly unimodular way. Detailed definitions can be found in \cite[Section 8]{AL07} and \cite[Section 3]{unimodular2}. Hyperfiniteness is closely related to amenability. Indeed, a unimodular transitive graph is hyperfinite if and only if it is amenable \cite[Theorems 5.1 and 5.3]{BLPS99}. A notion of amenability for unimodular random rooted graphs (sometimes referred to as \emph{invariant amenability}) was developed in \cite[Section 8]{AL07}, where it was shown to be equivalent to hyperfiniteness under the assumption that $\E[\deg(\rho)]<\infty$. See also \cite[Section 3]{unimodular2}. 
A unimodular random rooted \emph{tree} is hyperfinite if and only if it is either finite or has at most two ends almost surely; see \cite{unimodular2} for many further characterizations. In particular, a unimodular Galton-Watson tree with offspring distribution $\mu$ is hyperfinite if and only if $0 \leq \overline{\mu} \leq 1$.

The following theorem resolves \cite[Conjecture 4.2]{MR2914859}. (Note that a positive solution to that conjecture also follows from \cref{thm:ends}; the proof below is both more direct and more general.) We say that a quasi-unimodular random rooted graph $(G,\rho)$ is hyperfinite if its law is equivalent to that of a hyperfinite unimodular random rooted graph. (It follows from \cite[Theorem 8.5]{AL07} that if two unimodular random rooted graphs have equivalent laws, then one is hyperfinite if and only if the other is; note that the equivalence between the items of that theorem other than item 1 does not require the integrability assumption $\E[\deg(\rho)]<\infty$.)

\begin{thm}
\label{thm:hyperfinite}
Let $(G,\rho)$ be a unimodular random rooted graph, and let $(T,o)$ be an independent unimodular random rooted tree. Let $X$ be a $T$-indexed walk in $G$ with $X(o)=\rho$, and let $\operatorname{Tr}(X)$ be the trace of $X$. Suppose that $X$ is almost surely transient and that the integrability assumption $\E[\deg_G(\rho)]<\infty$ holds. Then $(\operatorname{Tr}(X),\rho)$ is hyperfinite if and only if $(T,o)$ is hyperfinite.
\end{thm}

Together with \cite[Theorem 8.15]{AL07}, \cref{thm:hyperfinite} has the following immediate corollary, which resolves \cite[Conjecture 4.1]{MR2914859}. 

\begin{corollary}
% Let $(G,\rho)$ be a unimodular random rooted graph with $\E[\deg(\rho)]<\infty$. 
Let $(G,\rho)$ be a bounded degree unimodular random rooted graph, and let $(T,o)$ be an independent ergodic unimodular random rooted tree. Let $X$ be a $T$-indexed walk in $G$ with $X(o)=\rho$, and let $\operatorname{Tr}(X)$ be the trace of $X$. Suppose that $X$ is almost surely transient.
 % and that the integrability assumptions $\E[\deg_G(\rho)]<\infty$ and $\E[\deg_G(\rho) \deg_{\operatorname{Tr}(X)}(\rho) (\# X^{-1}(\rho))^{-1}]<\infty$  both hold.
  If $(T,o)$ is not hyperfinite then simple random walk on $\operatorname{Tr}(X)$ has positive speed almost surely.
\end{corollary}

Here, a unimodular random rooted graph is said to be \textbf{ergodic} if the probability that it belongs to any re-rooting invariant event is in $\{0,1\}$, or, equivalently, if its law is an extreme point of the convex set of unimodular probability measures on $\cG_\bullet$ \cite[Theorem 4.7]{AL07}. This assumption is required to rule out, say, the case that $T$ is equal to $\Z$ with probability $1/2$ and is a $3$-regular tree with probability $1/2$.
It is not too hard to see that if $\mu$ is an offspring distribution with $\overline{\mu}>1$ then a unimodular Galton-Watson tree with offspring distribution $\mu$ conditioned to be infinite is ergodic, see \cite{MR1336708}.
 (The result can also be  applied in the non-ergodic case by invoking the existence of the \emph{ergodic decomposition}.) Note that the assumption that $G$ has bounded degrees is
  % $\E[\deg_G(\rho) \deg_{\operatorname{Tr}(X)}(\rho) (\# X^{-1}(\rho))^{-1}]<\infty$ is 
  needed to apply \cite[Theorem 8.15]{AL07} to the law of $(\operatorname{Tr}(X),\rho)$ biased by the weight $\deg_G(\rho)(\#X^{-1}(\rho))^{-1}$. (It is possible to weaken this assumption in various ways. In particular, it follows by well-known arguments that it suffices to assume that $G$ has at most exponential growth almost surely and that $\E[\deg_G(\rho)\deg_{\operatorname{Tr}(X)}(\rho)(\# X^{-1}(\rho))^{-1}]< \infty$. We do not pursue this here.)

\medskip
\noindent
\textbf{Soficity.}
Recall that every finite connected graph can be made into a unimodular random rooted graph by choosing the root uniformly at random.
 % and that the set of probability measures of unimodular random rooted graphs is a closed subset of the space of all probability measures on $\cG_\bullet$ with respect to the weak topology.
  A unimodular random rooted graph $(G,\rho)$ is said to be \textbf{sofic} if there exists a sequence of almost surely finite unimodular random rooted graphs $(G_n,\rho_n)$ converging in distribution to $(G,\rho)$. It is a major open problem whether every unimodular random rooted graph is sofic \cite[Section 10]{AL07}; this is of particular interest when $G$ is the Cayley graph of a finitely generated group.

This problem is well-understood for unimodular random \emph{trees}. Indeed, it is known that every unimodular random rooted tree is not only sofic but \emph{strongly sofic}, which roughly means that if we decorate the vertices and edges of the tree in an arbitrary unimodular way then the resulting decorated tree remains sofic. This was first proven for Cayley graphs of free groups by Bowen \cite{Bowen03}, and was extended to arbitrary unimodular random rooted trees by Elek \cite{Elek10}; see also \cite{URT} for a more probabilistic approach. Strong soficity has better stability properties than soficity, and it can be deduced from these results that, roughly speaking, various unimodular random rooted graphs that can be equipped with some sort of tree structure are strongly sofic also. See \cite{ElekLipp10} and \cite[Theorem 2]{unimodular2} for precise results.

Using these ideas, it is quite straightforward to prove the following theorem, which answers positively \cite[Question 4.5]{MR2914859}. We say that a \emph{quasi-unimodular} random rooted graph is (strongly) sofic if some unimodular random rooted graph with equivalent law is (strongly) sofic. (Again, it can be proven that this does not depend on which equivalent law one chooses, but we will not need this.)

\begin{thm}
\label{thm:sofic}
Let $(G,\rho)$ be a unimodular random rooted graph, and let $(T,o)$ be an independent unimodular random rooted tree. Let $X$ be a $T$-indexed walk in $G$ with $X(o)=\rho$, and let $\operatorname{Tr}(X)$ be the trace of $X$. Suppose that $X$ is almost surely transient and that the integrability assumption $\E[\deg_G(\rho)]<\infty$ holds. Then $(\operatorname{Tr}(X),\rho)$ is strongly sofic.
\end{thm}

We now sketch a proof of \cref{thm:hyperfinite,thm:sofic}. In the interest of space we have refrained from giving a self-contained exposition; the reader may find it helpful to read \cite[Sections 3 and 8]{unimodular2} before returning to the proof below.

\begin{proof}[Sketch of proof of \cref{thm:hyperfinite,thm:sofic}]
Let $(T,o)$, $(G,\rho)$, and $X$ be as in the statement of the theorems, and let $(H,\rho_H)$ be a unimodular random rooted graph whose law is given by biasing the law of $(\operatorname{Tr}(X),\rho)$ by $\deg(\rho)(\#X^{-1}(\rho))^{-1}$. In order to prove both theorems, it suffices by \cite[Proposition 3.12, Theorem 8.1, and Theorem 8.2]{unimodular2} to prove that the unimodular random rooted graphs $(T,o)$ and $(H,\rho_H)$ are \textbf{coupling equivalent}. This means that there exists a random quadruple $(F,\omega_1,\omega_2,\rho_F)$ such that the following conditions hold:
\begin{enumerate}
 \item $(F,\rho_F)$ is a unimodular random rooted graph.
 \item $\omega_1$ and $\omega_2$ are random connected subgraphs of $F$, encoded as functions $\omega_i: V(F) \cup E(F) \to \{0,1\}$ such that $\omega_i(v)=1$ for every $v\in V(F)$ such that $\omega_i(e)=1$ for some edge $e$ incident to $v$. (In particular, these subgraphs need not be spanning.)
 \item The quadruple $(F,\omega_1,\omega_2,\rho_F)$ is unimodular in an appropriate sense. (That is, as a random element of the space of rooted graphs decorated by two subgraphs. This space carries a natural variant of the local topology, and unimodular random elements of it are defined as before.)
 \item The conditional distribution of $(\omega_1,\rho_F)$ given that $\omega_1(\rho_F)=1$ is equal to the distribution of $(T,o)$, and the conditional distribution of $(\omega_2,\rho_F)$ given that $\omega_2(\rho_F)=1$ is equal to the distribution of $(H,\rho_H)$.
\end{enumerate}
Such a quadruple $(F,\omega_1,\omega_2,\rho_F)$ is referred to as a \textbf{unimodular coupling} of $(T,o)$ and $(H,\rho_H)$. Coupling equivalence was introduced in \cite{unimodular2} and is closely related to the notion of \emph{measure equivalence} in group theory.

We now construct such a unimodular coupling. Let $(G',\rho')$ be a random rooted graph whose law is given by biasing the law of $(G,\rho)$ by $\deg(\rho)$. Let $X'$ be a $T$-indexed random walk on $G'$ with $X'(o)=\rho'$.
 Let $(F,o)$ be the random rooted graph with the same vertex set as $(T,o)$ and where the number of edges connecting two vertices $u$ and $v$ is equal to the number of edges connecting $u$ and $v$ in $T$ plus the number of edges connecting $X'(u)$ and $X'(v)$ in $\operatorname{Tr}(X)$. Thus, the edge set of $F$ can naturally be written as a disjoint union $E_1 \cup E_2$, where $E_1$ is equal to the edge set of $T$. We let $\omega_1$ be the subgraph of $F$ that contains every vertex and that contains exactly those edges of $F$ that belong to $E_1$. For each vertex $x$ in the trace of $X'$, let $\phi(x)$ be a uniformly random element of the finite set $(X')^{-1}(x)$, and let $\Phi$ be the set of vertices of $F$ that are equal to $\phi(x)$ for some $x$ in the trace of $X$. We let $\omega_2$ be the subgraph of $F$ with vertex set $\Phi$ and with edge set the set of edges that belong to $E_2$ and have both endpoints in $\Phi$. It follows by a similar proof to that of \cref{prop:unimodularpush} that $(F,\omega_1,\omega_2,o)$ is unimodular. Moreover, we trivially have that $(\omega_1,o)$ is equal to $(T,o)$, and can easily verify that the conditional distribution of $(\omega_2,o)$ given that $\omega_2(o)=1$ (i.e., that $o\in \Phi$) is equal to the distribution of $(H,\rho_H)$. Indeed, $\omega_2$ is clearly isomorphic to the trace of $X'$ and, since $\phi(\rho')$ is uniform on $(X')^{-1}(\rho')$, conditioning on $o \in \Phi$ has the same effect as biasing by $(\# (X')^{-1}(\rho'))^{-1}$; we omit the details.
\end{proof}

% \begin{thm}
% \label{thm:further}
% Let $(G,\rho)$ be a unimodular random rooted graph, and let $(T,o)$ be an independent unimodular random rooted tree. Let $X$ be a $T$-indexed walk in $G$ with $X(o)=\rho$, and let $\operatorname{Tr}(X)$ be the trace of $X$. Suppose that $X$ is almost surely transient and that $\E[\deg(\rho)]<\infty$.  Then the following hold:
% \begin{enumerate}
% \item $(\operatorname{Tr}(X),\rho)$ is quasi-unimodular with weight
% \[
% W(\operatorname{Tr}(X),\rho) = \myfrac{\E[\deg(\rho)(\#X^{-1}(\rho))^{-1} \mid (H,\rho)]}{\E[\deg(\rho)(\#X^{-1}(\rho))^{-1}]}
% \]
% \item $(\operatorname{Tr}(X),\rho)$ is hyperfinite if and only if $(T,o)$ is.
% \item $(\operatorname{Tr}(X),\rho)$ is sofic.
% \end{enumerate}
% \end{thm}

%  Combining  this result with \cite[Theorem 8.15]{AL07} yields the following immediate corollary, which resolves \cite[Conjecture 4.1]{MR2914859}.

\section*{Acknowledgments}

We thank Itai Benjamini, Jonathan Hermon, Asaf Nachmias, and Elisabetta Candellero for useful discussions. In particular, we thank Asaf for discussions that led to a substantially simpler proof of \cref{thm:magic}. We also thank the anonymous referee for their careful reading and helpful suggestions.

 \setstretch{1}
 \footnotesize{
  \bibliographystyle{abbrv}
  \bibliography{unimodularthesis.bib}
  }

\end{document}

%% file: header.tex
\usepackage{amsmath,amssymb,amsfonts,amsthm}
\usepackage{color}

\usepackage[colorlinks=true,linkcolor=blue,citecolor=blue,urlcolor=blue,pdfborder={0 0 0}]{hyperref}
\usepackage{cleveref}

\usepackage{caption}
\usepackage{thmtools}

\usepackage{mathrsfs}

\crefname{theorem}{Theorem}{Theorems}
\crefname{thm}{Theorem}{Theorems}
\crefname{lemma}{Lemma}{Lemmas}
\crefname{lem}{Lemma}{Lemmas}
\crefname{remark}{Remark}{Remarks}
\crefname{prop}{Proposition}{Propositions}
\crefname{defn}{Definition}{Definitions}
\crefname{corollary}{Corollary}{Corollaries}
\crefname{conjecture}{Conjecture}{Conjectures}
\crefname{question}{Question}{Questions}
\crefname{chapter}{Chapter}{Chapters}
\crefname{section}{Section}{Sections}
\crefname{figure}{Figure}{Figures}

\theoremstyle{plain}
\newtheorem{thm}{Theorem}[section]

\newtheorem{lemma}[thm]{Lemma}
\newtheorem{theorem}[thm]{Theorem}

\newtheorem{corollary}[thm]{Corollary}

\newtheorem{prop}[thm]{Proposition}

\theoremstyle{definition}

\theoremstyle{remark}
\newtheorem{remark}[thm]{Remark}

\numberwithin{equation}{section}

%% file: commands.tex
\renewcommand{\P}{\mathbb P}
\newcommand{\E}{\mathbb E}

\newcommand{\R}{\mathbb R}
\newcommand{\Z}{\mathbb Z}
\newcommand{\N}{\mathbb N}

\newcommand{\cF}{\mathcal F}
\newcommand{\cG}{\mathcal G}

\newcommand{\cL}{\mathcal L}

\newcommand{\cT}{\mathcal T}
\newcommand{\cU}{\mathcal U}

\newcommand{\sA}{\mathscr A}